\documentclass[12pt,a4paper]{amsart}
\usepackage{amssymb,amscd}
\usepackage{mathrsfs}
\usepackage[mathcal]{eucal}
\usepackage{float}
\usepackage[stable]{footmisc}

\usepackage{xy}
\input xy
\xyoption{all}
\usepackage{pdflscape}
\usepackage{hyperref}

\usepackage[english]{babel}

\def\A{\mathbb{A}}
\def\B{\mathbb{B}}
\def\Q{\mathbb{Q}}
\def\R{\mathbb{R}}
\def\C{\mathbb{C}}
\def\F{\mathbb{F}}

\let\myacute=\'

\def\emptyset{\varnothing}

\def\<{\langle}
\def\>{\rangle}

\def\Z{\mathbb{Z}}

\def\cL{\mathcal{L}}
\def\cM{\mathcal{M}}
\def\cN{\mathcal{N}}
\def\cS{\mathcal{S}}
\def\cP{\mathcal{P}}
\def\cB{\mathcal{B}}
\def\cC{\mathcal{C}}

\def \begindm {\begin{displaymath}}

\def \enddm {\end{displaymath}}

\def\C{\mathbb{C}}
\def\R{\mathbb{R}}
\def\Q{\mathbb{Q}}
\def\F{\mathbb{F}}

\def\cL{\mathcal{L}}

\def\cM{\mathcal{M}}
\def\cO{\mathcal{O}}

\newtheorem{thm}{Theorem}[section]
\newtheorem{lemma}{Lemma}[section]
\newtheorem{cor}{Corollary}[section]
\newtheorem{prop}{Proposition}[section]
\newtheorem{claim}{Claim}[section]

\numberwithin{equation}{section}

\long\def\symbolfootnote[#1]#2{\begingroup\def\thefootnote{\fnsymbol{footnote}}\footnote[#1]{#2}\endgroup}

\title{Decidability Problems for Adele Rings and related restricted products}
\author[J. Derakhshan]{Jamshid Derakhshan}
\address{St Hilda's College, University of Oxford, Cowley Place, Oxford OX4 1DY, and Mathematical Institute, Oxford, OX2 6GG, UK}
\email{derakhsh@maths.ox.ac.uk}

\author[A. Macintyre]{Angus Macintyre${}^{\dag}$}
\address{Queen Mary, University of London,
School of Mathematical Sciences, Queen Mary, University of London, Mile End Road, London E1 4NS, UK}
\email{angus@eecs.qmul.ac.uk}

\thanks{${}^{\dag}$Supported by a Leverhulme Emeritus Fellowship}

\begin{document}


\subjclass[2000]{}

\begin{abstract} 
We study elementary equivalence of adele rings and decidability for adele rings of general number fields. We prove that elementary equivalence of adele rings (of two number fields)
implies isomorphism of the adele rings.

\end{abstract}

\maketitle

\section{\bf Introduction}\label{sec-introduction}

In \cite{DM-ad},\cite{DM-supp}, technology from the classic paper \cite{FV} by Feferman-Vaught is combined with the classic  work of Ax-Kochen-Ersov to give quantifier eliminations for certain rings which are restricted products of certain Henselian fields with respect to their valuations. There are corresponding decidability results, but no systematic study of decidability problems. The theory applies to all adele rings $\A_K$ over number fields, but applies to many other related structures. There are serious problems in determining what is special, model-theoretically, about the adele rings $\A_K$.

The decidability of $\A_{K}$, where $K$ is a number fied, has been known since Weispfenning \cite{weisp-hab}, but there have been no explicit discussion of axioms or uniformity (in $K$). This is rectified in the present paper.

\cite{FV} applies to a huge range of structures associated to products of structures. "Associated" typically means definable or interpretable in the appropriate product structure. For us in \cite{DM-ad},\cite{DM-supp} the main examples were the adele rings of number fields, and there we combined the very general quantifier-elimination of \cite{FV} with classical model-theoretic work on Henselian fields, and classic model-theoretic work on Boolean algebras (Tarski, Feferman-Vaught) to get purely ring-theoretic elimination theory in the adele rings and set some foundations for adelic measures and integrals in adele geometry. In \cite{DM-ad} extensions of the Boolean results of Tarski-Feferman-Vaught that we proved in \cite{DM-Bool} were applied to get elimination theory in a language stronger than the ring language in relation to reciprocity laws in number theory. We paid little attension to decidability, and we now turn to this aspect, where \cite{FV} provides very powerful results.

Among the problems considered below are:

(i) Decidability of the individual $\A_K$, $K$ a number field,

(ii) Uniform decidability of the class of all $\A_K$,

(iii) Isomorphism and elementary equivalence for adele rings.

\section{\bf The generalized product of Feferman-Vaught}\label{fv}

For applications to decidability it is useful to spell out the essential data of the generalized product construction.

{\bf Item 1} The first-order language $L_{Boolean}$ of Boolean algebras, with signature $\{\wedge,\vee,\neg,0,1\}$.

{\bf Item 2} For sets $I$, the atomic Boolean algebras $Pow(I)$, the powerset of $I$.

{\bf Item 3} An expansion $L_{Boolean+}$ of $L_{Boole}$.

A crucial example is got by adding a unary predicate $Fin(x)$, to be interpreted in $Pow(I)$ by the ideal of finite subsets of $I$. 

{\bf Item 4} A first-order language $L$ disjoint from $L_{Boolean+}$.

{\bf Item 5} Let $\mathcal{F}$ be a family $(\mathcal{M}_i)_{i\in I}$ of $L$-structures and $\prod_{i\in I} \mathcal{M}_i$ the product. Let $\Theta(w_1,\dots,w_k)$ be an $L$-formula. Then, for $f_0,\dots,f_k \in \prod_{i\in I} \cM_i$, we define 
$$[[\theta(f_0,\dots,k_k)]]=\{i\in I: \cM_i\models \theta(f_0(i),\dots,f_k(i)\}.$$
For $(\cM_i)_{i\in I}$ fixed, we have the function $[[\Theta]]$ from $(\prod_{i\in I}\cM_i)^{k+1}$ to $Pow(I)$ given by 
$$[[\Theta]](f_0,\dots,f_k)=[[\Theta(f_0,\dots,f_k)]].$$
This observation is just to stress the functorial nature of the construction.

{\bf Item 6} (Continuing from Item 5) If $\Psi(v_0,\dots,v_k)$ is an $\cL_{Boole+}$-formula, and 
$$\bar{\theta}=(\theta_0(w_0,\dots,w_{k-1}),\dots,\theta_n(w_0,\dots,w_{k-1})),$$
we define $\Psi(\bar{\theta})$ as the $k$-ary relation on $\prod_{i\in I} \cM_i$ consisting of the $(f_0,\dots,f_{k-1})$ such that 
$$Pow(I)^+ \models \Psi([[\theta_0(f_0,\dots,f_{k-1})]],\dots,[[\theta_n(f_0,\dots,f_{k-1})]]).$$
Note that $\Psi(\bar{\theta})$ is well-defined provided $\Psi$ has arity $n$, each $\theta_i$ has arity $k$, and $\bar{\theta}$ has length $n$. If $L_{Boole+}$ and $L$ are computably given, then the pairs $(\Psi,\bar{\theta})$ with $\Psi(\bar{\theta})$ well-defined are computably given.

{\bf Item 7} (Continuing Item 6) Expand $L$ by all $k$-ary relations $\Psi(\bar{\theta})$ as above. This gives an expansion 
$L^{Boole+}$ of $L$. Note that is this independent of $Pow(I)^+$. However, from $Pow(I)^+$ we get an $L^{Boole+}$-structure on the product $\prod_{i\in I} \cM_i$. 

{\bf Item 8} The big theorem is that as $L^{Boole+}$-structure $\prod_{i\in I} \cM_i$ has quantifier-elimination, 
and moreover computably (if $L_{Boole+}$ amd $L$ are given computably).

In \cite{FV} the results are stated a bit differently, in terms of acceptable and partitioning sequences. This makes the proofs a bit more direct, but the two formalisms are effectively equivalent.

{\bf Item 9} The preceding now yields powerful preservation theorems. The first yields a preservation result for elementary equivalence. One starts as usual with with fixed $L_{Boole+}$ and $L$, and a set $I$. Two families of $L$-structures $(\cM_i)_{i\in I}$ and $(\cN_i)_{i\in I}$ are given, with $\cM_i \equiv \cN_i$ for all $i$ as $L$-structures. Then $\prod_{i\in I} \cM_i \equiv \prod_{i\in I} cN_i$ as $L^{Boole+}$-structures (cf. \cite{FV}, Theorem 5.1, page 76).

{\bf Item 10}. (Notation as preceding) Suppose $\cM_i \rightarrow \cN_i$ are elementary, then (easily) 
$\prod_{i\in I} h_i: \prod_{i\in I} \cM_i \rightarrow \prod_{i\in I} \cN_i$ is an $L^{Boole+}$-embedding, and is 
$L^{Boole+}$-elementary (cf. \cite{FV}, Theorem 5.2, page 77).

{\bf Item 11}. Let $\cS$ be a fixed class of $L_{Boole+}$-structures on some $Pow(I)$ (maybe a single $I$, maybe several). 
Let $\cC$ be a class of $L$-structures, and $\cB$ an $L_{Boole+}$-structure (in $\cS$, which is understood), living on $Pow(I)$. Let $(\cM_i)_{i\in I}$ be a family of structures from $\cC$. Then one defines $\cP((\cM_i)_{i\in I},\cB)$ as the generalized product defined earlier (in Item 7).

We write $\cP(\cC,\cB)$ for the class of all auch products.

{\bf Item 12}. (Continue notation) This concerns Theorem 5.6 of \cite{FV}.

\begin{thm}\cite[Theorem 5.6]{FV} If $\cS$ is decidable and $\cC$ is decidable, then so is $\cP(\cC,\cB)$.\end{thm}
This is extremely powerful, when combined with method of interpretations.

\section{\bf Decidability Theorems for Restricted Products}

Here $L$ is the language of rings. $L_{Boole+}$ is got from $L$ by adding a unary predicate $Fin(x)$, which is intended, in structures $Powerset(I)$ to stand for the filter of finite sets. It has long been known that if $I_1$ and $I_2$ are infinite, ,then 
$(Powerset(I_1),Fin) \equiv (Powerest(I_2),Fin)$, while for finite $I$ the theory of $(Powerset(I),Fin)$ is of course just the theory of the finite $Powerset(I)$. An explicit set of axioms (axioms for atomic Boolean algebras) is known. This set is recursive and its complete extensions is given by the $min(card(I),\infty)$. The theory of atomic Boolean algebras is decidable, and have a simple uniform quantifier elimination.

For infinite $I$, these will be our main example of $\cB$'s. The class $\cC$ of $L$-structures we study will be subclasses of the class of all finite extensions of the fields $\Q_p$ as $p$  varies, together with $\C$ and $\R$. The latter class is not known to be decidable, though the class of all $\Q_p$ together with $\C$ and $\R$ is decidable by deep work of Ax \cite{ax} (supplemented by classical work of Tarski, see \cite{KK}).

So if we take $\cC_0$ to be the class consisting of each $\Q_p$ and $\R$, then $\cP(\cC_0,\cB)$ is decidable by Feferman-Vaught \cite{FV}. This result is of no particular interest, but leads to more dramatic results.

It is a non-trivial fact \cite{CDLM} that there is a uniform first-order definition (in $L$) of the valuation rings $\cO_p$ as $p$ varies. Indeed there is a uniform first-order definition of the unit balls of the $\Q_p$ and $\R$, say by a formula $UnitBall(x)$. Now consider inside any element of $\cP(\cC_0,\cB)$ the definable subset given by $Fin([[\neg UnitBall(w)]])$. This is the set of all $f$ in the product which take values in the unit ball at all but finitely many $i\in I$. In the case above this is a subring. In the special case of the product $(\prod_{p} \Q_p) \times \R$ this is the ring $\A_{\Q}$ of adeles over $\Q$, a special case of the restricted product. Note that the decidability for $\cP(\cC_0,\cB)$ has no implications for the decidability of the individual adele ring $\A_{\Q}$. However, the proof of Theorem 5.6 in \cite{FV} (stated above in Item 12) shows much more than the statement says. The proof works locally, that is when there is just one family $(\cM_i)_{i\in I}$ and a $\cB$ and gives a quantifier elimination in the language of generalized products $L^{Boole+}$. Combined with the work of Ax \cite{ax} and Ax-Kochen this is enough to give

\begin{thm} The theory of $\A_{\Q}$ is decidable.\end{thm}
\begin{proof} See \cite{DM-ad} for details.\end{proof}

How to handle $\A_K$ for $K$ an arbitrary number field? The index set is now finitely branching over $Primes \cup \{\infty\}$. 
Aside from the real and complex factors, the product involves $K_{\cP_i}$ which contains $\Q_p$, where $\cP_i$ is over $p$ (of which there are finitely many for each $p$), and one has to figure out $K_{\cP_i}$ in terms of the residue field degree $f_i$ and ramification $e_i$. For this we appeal to the basic work of Kummer.

Pick $\alpha$ an algebraic integer, so that $K=\Q(\alpha)$. Let $f(x)$ be the monic minimal polynomial of $\alpha$ over $\Q$. Then $f\in \Z[x]$. 

Now fix a prime $p$ in $\Z$. Then $p$ lifts to finitely many primes $\cP_1,\dots,\cP_g$ in $\cO_K$, and 
$$p\cO_K=\cP_1^{e_1}\dots \cP_g^{e_g}$$
(usual prime decomposition). We put $e(\cP_i/p)=e_i$, and call it the ramification index of $\cP_i$ over $p$. Also, 
$\cO_K/\cP_i$ is a finite extension of $\F_p$ of dimension $f(\cP_i/p)$ over $\F_p$. The number $f_i$ is the residue degree of $\cP_i$ over $p$ In fact, if $K_{\cP_i}$ is the completion of $K$ at $\cP_i$, then $\Q_p\subseteq K_{\cP_i}$, as valued fields 
(both Henselian) and $e_i$ and $f_i$ are respectively the usual remification index and residue field degree.

Finally,
$$\sum_i e_if_i=[K:\Q].$$

If $[K:\Q]$ is given this already provides crude bounds for the $e_i$ and $f_i$ (i.e. $e_i,f_i\leq [K:\Q]$). The prime $p$ is unramified if $e_i=1$ for all $i$, and otherwise ramified. 

An important point is that if $p$ does not divide $N_{K/\Q}(f'(\alpha))$ then one has another way (uniform in $p$) for seeing the decomposition. It turns out that the above condition defines being ramified.

\begin{thm} For unramified $p$ the decomposition of $p\cO_K$ (up to order) is obtained as follows. Suppose 
$$f(x)=\prod_{i=1}^{n}g_i(x)^{e_1} ~ mod ~ p\cO_K[x]$$
where $e_i\geq 1$ and the $g_i$ are monic over $\Z$ such that their reductions modulo $p$ are irreducible and distinct. 
Then $\cP_i=p\cO_K+g_i(\alpha)\cO_K$ is a maximal ideal of $\cO_K$, for each $i$, and $p\cO_K=\cP_1^{e_1}\dots \cP_g^{e_g}$.\end{thm}

\

3.3. Our analysis of decidability will be intertwined with our analysis of axioms for adele rings. The analysis of definability in \cite{DM-ad},\cite{DM-supp} works uniformly. We study the adele rings $\A_K$ as special cases of the restricted direct products of certain Henselian valued fields with respect to a uniform definition of their valuation rings. We interpret 
(uniformly for $\cC,~\cB$ as spelled out earlier) the whole formalism ring-theoretically. The index set $I$ of absolute values of $K$ is coded as the Boolean algebra of minimal idempotents, and the "stalks" are $R/(1-e)R$, which are assumed to be models of the theory of an appropriate class $\cC$ of Henselian valued fields. Note that if $\cC$ consists of the $\Q_p$ and $\R$ and $I$ is an infinite set, and $R$ is the restricted direct product (with respect to $\cO$ as usual), and $L_i$ are chosen from $\cC$, there will be many different elementary theories of restricted products. One way to get alternatives is to have more than idempotent $e$ where $R/(1-e)R$ has residue characteristic $p_e$. Another is to have no $e$ with having 
residue characteristic $p$.

For $\A_{\Q}$ there is exactly one stalk of residue characteristic $p$, with $p$ unramified, and with residue field $\F_p$, and one stalk that is $\R$. These are elementary conditions.

\begin{lemma}\label{lemma1} If $K$ is a number field that is not $\Q$, there is no stalk with residue field $\F_p$ which is the unique stalk of residue characteristic $p$, and $p$ is unramified.\end{lemma}
\begin{proof} Fix $K$ and such a $p$. Use $\alpha\in \cO_K$ with $K=\Q(\alpha)$, and as before let $f$ be minimal (monic and over $\Z$) polynomial of $\alpha$ over $\cO$.  Let $n=[K:\Q]$. Let $\cP_1,\dots,\cP_g$ be the pries of $\cO_K$ above $p$. Let 
$e_i,f_i$ br the usual ramification index and residue degree. Since $p$ is unramified, all $e_i$ are $1$, so 
$$n=\sum_{i=1}^{g} f_i.$$
Bur if $n>1$, then either $g>1$ or $g=1$ and $f_i>1$ for some $i$.\end{proof}
\begin{cor} If $K\neq \Q$, then $\A_{\Q}$ and $\A_K$ are not elementarily equivalent.\end{cor}
\begin{proof} Choose $p$ unramified in $K$ and use Lemma \ref{lemma1}.\end{proof}

3.4. This suggests a number of questions:

(A) How are the $\A_K$ divided into elementary equivalence classes?

(B) To $\A_K$ we attach the set $\{p,e_1,\dots,e_g,f_1\dots,f_g\}$ as above, together with

(i) number of real valuations

(ii) number of complex valuations.

Does this process give a complete set of elementary invariants of $\A_K$?

(C) What are the constraints on the above conjectural invariants?

These will be addressed in the rest of the paper. We point out now that if has been known for some time that 
$\A_{K_1} \cong \A_{K_2}$ does not imply $K_1\cong K_2$. We will consider what happens when we replace $\cong$ by
$\equiv$.

3.5. Let $R$ be any restricted direct product of finite extensions of the $\Q_p$, as well as $\R$ and $\C$, with respect to the 
(uniformly definable) unit balls. We analyze $R$ ring-theoretically using the apparatus of the Boolean algeba of idempotent (as in \cite{DM-ad}). Thus the minimal idempotents $e$ have attached "stalks" $R/(1-e)$, which come from the collection of 
Henselian fields abbove, and $\R$ and $\C$.

For each prime $p$ (in $\Z$) we let $R_p$ be $\{e: R/(1-e) \text{~has~residue~characteristic}~p\}$. Now we have in adele rings a uniform definition of 
finite idempotents (see \cite{DM-ad}), and in any $\A_K$ the set $R_p$ is finite.

Moreover, for each $\A_K$ the earlier work with $e_i,f_i$ gives a bound, independent of $p$, for $card(R_p)$. One bound is 
$[K;\Q]$ and the issue immediately arises of interpreting $[K;\Q]$ in $\A_K$. Of course, by decidability of $\A_{\Q}$ one can certainly not interpret $\Q$ (diagonally embedded) in $\A_{\Q}$ (and same for an algebraic number field $K$ using work of Julia Robinson \cite{jrobinson} showing undecidability of $\cO_K$.

Let $R=\A_K$. How do we detect $[K;\Q]$ inside $R$? We use the concept $p$ splits completely in $K$ meaning that $p$ is unramified and all $f_i$ are $1$.

For such $p$, all $\cP_i$ have residue field $\F_p$ (i.e. by the Kummer argument, $f$ splits completely in $\F_p$). Now consider $L \supseteq K$, the normal closure of $K$. By Chebotarev's density Theorem infinitely many $p$ split 
completely in $L$ (much more is true and we will use that later). Let $g\in \Z[x]$ be the minimum polynomial 
for some $\beta \in \cO_L$, $L=\cO(\beta)$. Chose $p$ splitting completely. Then $g$ factors 
completely in $\F_p$. It follows that each $f_i$ is $1$. Thus $p$ splits completely in $K$ Thus $g=n=[K:\Q]$ ($p\cO_K=\prod_i \cP_i^{e_i}$ as usual).

So we define $[K:\Q]$ by taking any $p$ splitting completely, and then $n$ is the number of minimal idempotents $e$ 
with $R/(1-e)R$  having all residue fields $\F_p$ and $v(p)=1$. 

So the definition works for all but finitely many $p$.

\begin{cor} $\A_{K_1}\equiv \A_{K_2}$ implies that $[K_1:\Q]=[K_2:\Q]$.\end{cor}
\begin{proof} Done\end{proof}
Of course the converse fails, e.g. for $K_1=\Q(\sqrt{2}),~ K_2=\Q(\sqrt{3})$ since
$\sqrt{2} \notin \Q(\sqrt{3})$.

\

3.6. {\it Decomposition Patterns}. Here we follow Perlis \cite{perlis}. $p$ is not assumed unramified in $K$. The splitting type of 
$p$ in $K$ is a sequence $\Sigma_{p,K}=(f_1,\dots,f_g)$, where $f_1\leq \dots \leq f_g$, so that 
$p\cO_K=\cP_1\dots \cP_g$ ($\cP$'s distinct) and $f_j$ is the residue degree of $\cP_j$. Note that there can be 
repetitions. Note that the $e_i$ are not there. Clearly if $\A_K\equiv \A_L$, then for each $p$, 
$\Sigma_{p,K_1}=\Sigma_{p,K_2}$ (cf. previous discussion).

Define, for a splitting type $A$,
$$P_K(A)=\{p: \Sigma_{p,K}=A\}.$$
Clearly since $\sum_j f_j \leq [K:\Q]$, we have $P_K(A)=\emptyset$ for all but finitely many $A$. Now Perlis' splended 
Theorem 1 in \cite{perlis} gives immediately
\begin{thm}\label{theorem1} If $K_1$ and $K_2$ are arithmetically equivalent, then \begin{itemize}
\item $\zeta_{K_1}(s)=\zeta_{K_2}(s)$,
\item $K_1$ and $K_2$ have the same discriminant, 
\item $K_1$ and $K_2$ have the same discriminant,
\item $K_1$ and $K_2$ have the same number of real (resp. complex) absolte values,
\item $K_1$ and $K_2$ have the same normal closure,
\item The unit groups of $K_1$ and $K_2$.
\end{itemize} \end{thm}
\begin{proof} Immediate by Perlis' result \cite[Theorem 1,pp. 345]{perlis} since $P_{K_1}(A)=P_{K_2}(A)$ for all $A$.\end{proof}

\begin{cor} For any given number field $K$, there are only finitely many number fields $L$ such that are $K$ and $L$ have elementarily equivalent adele rings, more generally, finitely many $L$ which are arithmetically equivalent to $K$.\end{cor}
\begin{proof} Use Theorem \ref{theorem1} (ii) and Hermite's Theorem.\end{proof}

{\bf Note:} It may seem strange that the $e_i$ are not used. The same is true in Perlis \cite{perlis} up to the proof of his Theorem 1. But in fact the $e_i$ are at most $log_2[K:Q]$, and we also know discriminant, so we have the ramified primes.
But it can happen that the ramification degrees may not match up for arithmetically equivalent (i.e. same zeta) $K_1$ and $K_2$.

\

3.7. {\it Elementary equivalence and isomorphism of adele rings}. In this section we show that elementary equivalence of adele rings implies isomorphism. Given a number field $K$, denote the set of its non-archimedean valuations by $V_K^f$.


\begin{prop}\label{eq} Let $K$ and $L$ be number fields. The following are equivalent.
\begin{itemize}
\item $K$ and $L$ are arithmetically equivalent, i.e. for all but finitely many $p$, the decomposition types of $p$ in $K$ and $L$ coincide.
\item $K$ and $L$ have the same zeta function.
\item There is a bijection $\phi: V_{K}^f \rightarrow V_{L}^f$ such that the local fields $K_{\frak p}$ and 
$L_{\phi(\frak p)}$ are isomorphic for all but finitely many primes $\frak p \in V_{K}^f$.
\end{itemize}
\end{prop}
\begin{proof} These follows from III.1 and VI.2 in the book \cite{klingen}.
\end{proof}

The crucial tool used in the proof is the following result of K. Iwasawa.

\begin{thm}\label{iwasawa} $\A_K$ and $\A_L$ are isomorphic if and only there there is a bijection 
$\phi: V_{K}^f \rightarrow V_{L}^f$ such that the local fields $K_{\frak p}$ and 
$L_{\phi(\frak p)}$ are isomorphic for all primes $\frak p \in V_{K}^f$ if and only if the finite adeles 
$\A_K^{fin}$ and $\A_L^{fin}$ are isomorphic.\end{thm}
\begin{proof} This is proved in \cite{iwasawa} on pages 331-356.\end{proof}

We note that isomorphism of $K$ and $L$ implies that of $\A_K$ and $\A_L$ which inturn implies $\zeta_K=\zeta_L$. 
In general, these implications can not be reversed, but they can be if $L$ or $K$ is Galois over $\Q$.

We shall also need the following.

\begin{prop} Let $K$ and $L$ be finite extensions of $\Q_p$. Suppse $K$ has ramification index $e$ and residue degree $f$. 
Suppose that $s>(p/p-1 + v_p(e)e)$. Then if $\mathcal{O}_K/\pi^s \mathcal{O}_K$ and 
$\mathcal{O}_L/\pi^s \mathcal{O}_L$ are isomorphic, then $K$ and $L$ are isomorphic.\end{prop}
\begin{proof} This is Proposition 2.1(iii) on page 17 in \cite{keating}.
\end{proof}
 
Now we can prove the theorem on isomorphism of adele rings.
\begin{thm}\label{isom} Let $K$ an $L$ be number fields. If $\A_K$ and $\A_L$ are elementarily equivalent, then they are isomorphic. 
\end{thm}
\begin{proof}
Suppose that $\A_K$ and $\A_L$ are elementarily equivalent. 
Then for almost all $p$, the decomposition types of $p$ in $K$ and $L$ are the same, and $\zeta_K(s)=\zeta_L(s)$ for all $s$. By Proposition \ref{eq}, there is a bijection $\phi: V_{K}^f \rightarrow V_{L}^f$ such that the local fields $K_{\frak p}$ and 
$L_{\phi(\frak p)}$ are isomorphic for all but finitely many primes $\frak p \in V_{K}^f$. Let 
$S:=\{\frak p_1,\dots,\frak p_t\}$ denote the set of primes such that for all $\frak p$ outside $S$ there is an isomorphism of completions. 

By Theorem \ref{iwasawa} it suffices to prove that 
there is a bijection 
$\phi: V_{K}^f \rightarrow V_{L}^f$ such that the local fields $K_{\frak p}$ and 
$L_{\phi(\frak p)}$ are isomorphic for all primes $\frak p \in V_{K}^f$. 

We know that there is an isomorphism for the primes $\frak p$ outside $S$. So it suffices to prove that there is a bijection 
defined on primes from $S$
$$\frak p \rightarrow \psi(\frak p)$$ 
such that the corresponding completions $K_{\frak p}$ and $L_{\psi(\frak p)}$ are isomorphic.

We consider the (rational) primes $\Omega:=\{p_1,\dots,p_l\}$ such that the primes in $S$ lie over some $p_j\in \Omega$, $j\leq l$, and the corresponding $p$-adic fields $\Q_{p_1},\dots,\Q_{p_l}$.

For any $\frak p\in S$, consider the completion $K_{\frak p}$. Let $e$ and $f$ denote its ramification index and residue degree respectively. For any finite extension $F$ of $\Q_p$ for any $p$, by the uniform definition of the valuation ring in \cite{CDLM}  the valuation ring, its maximal ideal, and its uniformizer are defined independently of the field by a first-order formula. Therefore $\mathcal{O}_F/\pi^s\mathcal{O}_F$ can be expressed independently of the field $F$ for any $s$, and finite extension $F$ of $\Q_p$ for any $p$. 

On the other hand by finiteness, we can write down a sentence in the language of rings which characterizes the finite ring $\mathcal{O}_{K_{\mathfrak p}}/\pi^s\mathcal{O}_{K_{\mathfrak p}}$ up to isomorphism, for each of 
the primes $\frak p$ from $S$, and each $s$. Denote this sentence by $\Phi(s,\frak p)$. 

Now consider the sentence "for all idempotents $e$ with stalk $e\A_K$ of 
residue characteristic $p$ among the primes in $\Omega$, for some $\mathfrak p$, $\Phi(s,\frak p)$ holds in 
$\mathcal{O}/\pi^s\mathcal{O}$, where $\mathcal{O}$ and $\pi$ {\it denote the sets defined by 
the formulas defining these sets uniformly for all local fields}, and $s$ satisfies $s>(p/p-1+v_p(e)e)$.

Since $\A_K$ and $\A_L$ are elementarily equivalent, this sentence holds in $\A_L$. This means that for each completion 
$L_{\mathfrak p}$ of residue characterisic $p$, the ring 
$$\mathcal{O}_{L_{\mathfrak p}}/\pi^s\mathcal{O}_{L_{\mathfrak p}}$$
is isomorphic to the ring
$$\mathcal{O}_{K_{\mathfrak p'}}/\pi^s\mathcal{O}_{K_{\mathfrak p'}},$$
for some $\frak p'$ lying over $p$, where $p$ is from $\Omega$.

Since there is a bijection between the primes in $L$ and those of $K$, we deduce by Proposition 0.3 that any completion 
$L_{\mathfrak p}$ is isomorphic to a completion $K_{\mathfrak p'}$ for some prime $\mathfrak p'$ and vice versa.
In other words, there is a bijection defined on the primes in $S$ inducing an isomorphism on completions. The proof is complete.
\end{proof}

\

3.8. {\it Decidability of $\A_K$.} We have a direct product decomposition 
$$\A_K=\prod_{v ~ Archimidean} K_v \times \A_K^{fin}.$$
By Feferman-Vaught \cite{FV}, it suffices to show that $\A_K^{fin}=\prod_{V_K^{fin}} K_v$ is decidable. We have only to show that the theory of the class of all $\cO_{\cP}$ is decidable. Now we separate off the finitely many ramified $p$ and the finitely many ramified $\cO_{\cP}$ "above them". Again make product decomposition, and note the $\cO_{\cP}$ just 
mentioned are all finitely ramified. By \cite{} they and their product are decidable. Now we turn our attention to the complementary product of the unramified $\cO_{\cP}$. It suffices to prove this is decidable. Now given $K$, there are 
only finitely many non-empty splitting types $A$. Again factorize into various $\prod_{P_K(A)} ...$. To show these are decidable
again by factorizing by inertia, we reduce to the problem of decidability, for fixed $f$, of $\prod{P_K(a)}$, dimension $f$ unramified extensions of $\Q_p$.

\bibliographystyle{acm}
\bibliography{bibadeles}

\end{document}